\documentclass{article}

\usepackage{amsmath, amssymb, enumerate, amsthm, setspace}
\usepackage[all]{xy}

\newcommand{\Top}{\mathbf{Top}}

\newcommand{\Set}{\mathbf{Set}}

\newcommand{\Ord}{\mathbf{Ord}}

\newcommand{\Rel}{\mathbf{Rel}}

\newcommand{\Grph}{\mathbf{Grph}}
\newcommand{\Grp}{\mathbf{Grp}}

\newcommand{\C}{\ensuremath{\mathbb{C}}}
\newcommand{\D}{\ensuremath{\mathbb{D}}}

\newcommand{\M}{\ensuremath{\mathcal{M}}}

\newcommand{\op}{\mathrm{op}}
\newcommand{\Sub}{\mathrm{Sub}}

\newcommand{\Eq}[1]{\ensuremath{\mathrm{Eq}(#1)}}
\newcommand{\Ef}[1]{\ensuremath{\mathrm{Ef}(#1)}}
\newcommand{\Proj}[1]{\ensuremath{\mathrm{Proj}(#1)}}

\title{$\mathcal{M}$-coextensive objects and the strict refinement property}

\author{Michael Hoefnagel}

\newtheorem{theorem}{Theorem}
\numberwithin{theorem}{section}

\newtheorem*{theorem*}{Theorem}
\newtheorem*{proposition*}{Proposition}
\newtheorem{definition}{Definition}
\numberwithin{definition}{section}
\newtheorem{proposition}{Proposition}
\numberwithin{proposition}{section}

\numberwithin{lemma}{section}
\newtheorem{corollary}{Corollary}
\numberwithin{corollary}{section}
\newtheorem{example}{Example}
\numberwithin{example}{section}
\newtheorem{remark}{Remark}
\numberwithin{remark}{section}
\newtheorem{convention}{Convention}

\date{}
\begin{document}
\maketitle
\begin{abstract}
The notion of an $\M$-coextensive object is introduced in an arbitrary category $\C$, where $\M$ is a distinguished class of morphisms from $\C$. This notion allows for a categorical treatment of the strict refinement property in universal algebra, and highlights its connection with extensivity in the sense of Carboni, Lack and Walters. If $\M$ is the class of product projections in a category $\C$ with finite products, then $\M$-coextensivity is closely related to the notion of Boolean category in the sense of E. Manes: $\C$ is co-Boolean if and only if its product projections are pushout stable, and every object is $\M$-coextensive. We show that if $\C$ is a variety of algebras then the $\M$-coextensive objects are precisely those algebras which have the strict refinement property, when $\M$ is the class of product projections. If $\M$ is the class of surjective homomorphisms in the variety, then the $\M$-coextensive objects are those algebras which have directly-decomposable (or factorable) congruences. Moreover, these results are proved for any object with global support in a regular category. We also show that in exact Mal'tsev categories, every centerless object with global support is projection-coextensive, i.e., has the strict refinement property. We will also show that in every exact majority category, every object with global support has the strict refinement property. 
\end{abstract}
The published version of this article can be found at \cite{Hoefnagel2020b}. 
\section{Introduction} \label{sec-introduction}
%The \emph{strict refinement property} was first introduced in \cite{CJT64} for relational structures. For a universal algebra $X$ it may be formulated as follows: $X$ has the strict refinement property if for any two product diagrams $(X \xrightarrow{p_i} A_i)_{i \in I}$ and $(X \xrightarrow{q_j} B_j)_{j \in J}$,  morphisms $\alpha_{i,j}:A_i \rightarrow C_{i,j}$ and $\beta_{i,j}:B_{j} \rightarrow C_{i,j}$ for $i\in I$ and $j\in J$, such that $\alpha_{i,j} p_i = \beta_{i,j} q_j$, and the diagrams $(A_i \xrightarrow{\alpha_{i,s}} C_{i,s})_{s \in J}$ and $(B_j \xrightarrow{\beta_{t,j}} C_{t,j})_{t \in J}$ are product diagrams for every $i\in I$ and $j \in J$. Examples of universal algebras which have the strict refinement property include any unitary ring, lattice, implication algebra, centerless or perfect group. The strict refinement property implies that if $X$ decomposes as a product of directly irreducible objects, then this decomposition is uniquely determined up to isomorphism of the factors. The dual co-strict refinement property is shared by many categories of a geometric nature

In universal algebra, various refinement properties have been defined for direct-product decompositions of structures, all of which give information about the uniqueness of such decompositions. The so-called \emph{strict refinement property} was first defined in \cite{CJT64}, and implies that any isomorphism between a product of irreducible structures is uniquely determined by a family of isomorphisms between each factor. If $A$ is a universal algebra which has the strict refinement property, then it was proved in \cite{CJT64} that the \emph{factor-congruences} of $A$ (see Definition~\ref{def-factor-relation}) form a Boolean sublattice of $\Eq{A}$ - the lattice of congruences on $A$. Moreover, this is a characteristic property of algebras which have the strict refinement property. Examples of structures which possess the strict refinement property include any unitary ring, any centerless or perfect group, any connected poset or digraph \cite{Hashimoto1951, Walker1987}, any lattice, or more generally, any congruence distributive algebra. Many geometric structures possess the dual property, which is to say that they have the co-strict refinement property. As we will show in this paper, this is mainly due to the fact that categories of geometric structures tend to be \emph{extensive} in the sense of \cite{CLW93}. The main aim of this paper is to investigate the relationship between the strict refinement property and coextensivity. In particular, we introduce the notion of an $\M$-coextensive object in an arbitrary category $\C$, where $\M$ is a distinguished class of morphisms from $\C$. When $\M$ is the class of product-projections in a category with finite products, and in this case the $\M$-coextensive objects are called \emph{projection-coextensive}, then projection-coextensivity is equivalent to the strict refinement property for algebras in a variety. Projection-coextensivity is also closely related to the notion of a \emph{Boolean category} in the sense of E. Manes \cite{Manes1992}. This relationship is captured by Corollary~\ref{cor:Boolean-category} below: every object in a co-Boolean category is projection-coextensive. But there are categories with finite products in which every object is projection-coextensive which are not co-Boolean. However, if $\C$ is a category with finite products and pushout stable product projections, then $\C$ is co-Boolean if and only if every object in $\C$ is projection-coextensive. 

\subsection{The notion of $\M$-coextensivity}
\noindent
Recall that a category $\C$ with finite products is \emph{coextensive} if the canonical functor 
\[
(A\downarrow \C) \times (B\downarrow \C) \xrightarrow{\times} (A \times B \downarrow \C),
\]
 is an equivalence of categories. The following proposition provides an equivalent definition of coextensivity, which is what the notion of $\M$-coextensivity is based on.
 \begin{proposition*}[Dual of Proposition 2.2 in \cite{CLW93}]
  A category $\C$ with finite products is coextensive if and only if it admits pushouts of arbitrary morphisms along product projections, and in every commutative diagram
\[
\xymatrix{
	A_1 \ar[d] &  X\ar[r] \ar[d] \ar[l] & A_2  \ar[d] \\
	B_1  & \ar[r] Y  \ar[l] & B_2
}
\]
where the top row is a product diagram, the bottom row is a product diagram if and only if both squares are pushouts.
 \end{proposition*}
\noindent
\begin{definition} \label{def-M-pushout}
Let $\C$ be a category and $\M$ a class of morphisms from $\C$. A commutative square
\[
\xymatrix{
X \ar[r] \ar[d] & A \ar[d]^{a} \\
B \ar[r]_{b} & P
}
\] 
in $\C$ is called an \emph{$\M$-pushout} if it is a pushout in $\C$, and $a,b$ are morphisms in $\M$.
\end{definition}
 \begin{definition} \label{def-M-coextensive}
 	Let $\C$ be a category and $\M$ a class of morphisms in $\C$. An object $X$ is said to be $\M$-\emph{coextensive} if every morphism in $\M$ with domain $X$ admits an $\M$-pushout along every product projection of $X$, and in each commutative diagram 
 	\[
 	\xymatrix{
 		A_1 \ar[d] &  X\ar[r] \ar[d] \ar[l] & A_2  \ar[d] \\
 		B_1  & \ar[r] Y  \ar[l] & B_2
 	}
 	\]
  where the top row is a product diagram and the vertical morphisms belong to $\M$, the bottom row is a product diagram if and only if both squares are $\M$-pushouts. 
 \end{definition}
 Of particular importance to the current paper is when $\M$ is the class of all product projections in $\C$, and in this case we call an object \emph{projection-coextensive} if it is $\M$-coextensive. We show that every projection-coextensive object in a category with (finite) products has the (finite) strict refinement property, and when the base category is \emph{regular} \cite{BGO71}, the converse also holds. Every projection-coextensive object has epimorphic product projections, so that the full subcategory of $(X\downarrow \C)$ consisting of product projections is a preorder. In a coextensive category, the posetal-reflection of this preorder $\Proj{X}$ is a Boolean-lattice (since every coextensive category is co-Boolean in the sense of \cite{Manes1992}). We show that this result still holds for projection-coextensive objects, and that it is a characteristic property of projection-coextensivity (Theorem~\ref{thm-projection-coextensive-characterization}). This result can be seen as an analogue of Theorem~5.11 in \cite{Manes1992} and Proposition~1.3.3 in \cite{Diers1986} for objects in a category with finite products. It also captures the classical characterization of the strict refinement property for algebras mentioned earlier. \emph{Pre-exact} categories are defined in Definition~\ref{def-pre-exact} as an intermediate between regular and exact categories. In the pre-exact context we establish a characterization of the strict refinement property (Theorem~\ref{thm-characterization-two-strict-refinement}) which allows us to show that \emph{centerless objects} in a \emph{Mal'tsev category} \cite{CPP91, CLP91} have the strict refinement property, as well as that every object (with global support) in a pre-exact \emph{majority category} \cite{Hoe18a} has the strict refinement property. In particular, this shows that any object with global support in an \emph{arithmetical category} in the sense of \cite{Ped96} has the strict refinement property. 

When $\M$ is chosen to be the class of regular epimorphisms in $\C$, then we say that $X$ is \emph{regularly-coextensive}{} if it is $\M$-coextensive. We will show that if $\C$ is then a variety of universal algebras, then any non-empty algebra $X$ is regularly-coextensive if and only if it has \emph{factorable congruences} in the sense of \cite{I2000}. Moreover, this result extends to regular categories. It is then immediate that any object $X$ in a regular category with global support which has factorable congruences necessarily has the strict refinement property. This generalizes a result of \cite{I2000}. 

\begin{convention}
Throughout this paper, we will assume that categories have finite products, so that by ``a category $\C$'', we mean ``a category $\C$ with finite products''. If $f: X \rightarrow Y$ and $g:Y \rightarrow Z$ are morphisms in a category $\C$, then we will always denote their composite by $gf$.
\end{convention}
\section{$\M$-coextensive objects have strict refinements}
Throughout this section, suppose that $\C$ is a category and that $\M$ is a class of morphisms in $\C$. Unless stated otherwise, we will assume that $\M$ contains all isomorphisms in $\C$, and is closed under composition and products in $\C$. 
\begin{remark} \label{rem-prod-projection-in-M}
Under these assumptions on $\M$, any product projection of an $\M$-coextensive object is contained in $\M$. This is because we may take the vertical morphisms in the diagram of Definition~\ref{def-M-coextensive} to be the identity morphisms. Note also that if $X$ is $\M$-coextensive, then for any product projection $X \xrightarrow{p} A$ we have that any terminal morphism $A \rightarrow 1$ is in $\M$. This is because in the diagram:
\[
\xymatrix{
	X \ar[d]_{p} &  X\ar[r] \ar[d]_p \ar[l]_{1_X} & 1  \ar[d] \\
	A  & \ar[r] A  \ar[l]^{1_A} & 1
}
\]
the bottom row is a product diagram, all the vertical morphisms are in $\M$, and hence the right-hand square is an $\M$-pushout by $\M$-coextensivity of $X$.
\end{remark}
\noindent
If $X \xrightarrow{\pi_1} X_1$ is a product projection, then a morphism $X \xrightarrow{\pi_2} X_2$ is called a \emph{complement} of $\pi_1$ if the diagram
\[
X_1 \xleftarrow{\pi_1} X \xrightarrow{\pi_2} X_2,
\]
is a product diagram in $\C$. Recall that $X$ is said to have \emph{co-disjoint} products if any product projection of $X$ is an epimorphism and the pushout of any complementary product projections of $X$ is  terminal. The following proposition is essentially just the dual of Proposition~2.6 in \cite{CLW93}, when coextensivity is restricted to single objects. 
\begin{proposition} \label{prop-product-projections-epimorphisms} Let $X$ be an $\M$-coextensive object in a category $\C$, then $X$ has co-disjoint products. 
\end{proposition}
\begin{proof}
	Consider the diagram below, where the top row is a product diagram.
	\[
	\xymatrix{
		A \ar[d] & \ar[l]_{\pi_1} X \ar[d]^{\pi_2} \ar[r]^{\pi_2} & B \ar[d]^{1_B} \\
		1 & \ar[l] B \ar[r]_{1_B} & B
	}
	\]
	By Remark~\ref{rem-prod-projection-in-M}, all the vertical morphisms are in $\M$. Since the bottom row is a product diagram, it follows that both squares are pushouts. This implies that $\pi_2$ is an epimorphism, and that the pushout of $\pi_1$ along $\pi_2$ is a terminal object.
\end{proof}
\begin{proposition} \label{prop-composite-proposition}
Let $X$ be an $\M$-coextensive object in $\C$, and suppose that $X\xrightarrow{p} A$ is any product projection. Then for any morphism $f:A \rightarrow B$ in $\C$, if $f   p \in \M$ then $f \in \M$. 
\end{proposition}
\begin{proof}
There exists an $\M$-pushout diagram of $p$ along $f   p$ isomorphic to the pushout diagram:
\[
\xymatrix{
	A \ar[d]_{f} & X \ar[l]_{p} \ar[d]^{f  p} \\
	B & B \ar[l]^{1_B}
}
\]
Since $\M$ contains all isomorphisms and is closed under composition, it follows that $f$ is contained in $\M$. 
\end{proof}

\begin{definition} \label{def-projection-coextensive}
	An object $X$ in a category $\C$ is called \emph{projection-coextensive} if it is $\M$-coextensive with $\M$ the class of all product projections in $\C$.
\end{definition}
\begin{remark}
Note that Remark~\ref{rem-prod-projection-in-M} implies that every object in $\C$ that is $\M$-coextensive, is necessarily projection-coextensive. 
\end{remark}
\noindent
The strict refinement property mentioned in the introduction may be formulated as follows: 
\begin{definition} \label{def-strict-refinement}
	An object $X$ in a category $\C$ is said to have the (finite) \emph{strict refinement property} if for any two (finite) product diagrams $(X \xrightarrow{f_i} A_i)_{i \in I}$ and $(X \xrightarrow{g_j} B_j)_{j \in J}$, there exist families of morphisms $(A_i \xrightarrow{\alpha_{i,j}} C_{i,j})_{i\in I,j\in J}$ and \\$(B_j \xrightarrow{\beta_{i,j}} C_{i,j})_{i\in I, j\in J}$ such that $\alpha_{i,j}f_i =\beta_{i,j}g_j$ and the diagrams $(A_i \xrightarrow{\alpha_{i,s}} C_{i,s})_{s\in J}$ and $(B_j \xrightarrow{\beta_{t,j}} C_{t,j})_{t \in I}$ are product diagrams for any $i\in I$ and $j\in J$. 
\end{definition} 
\begin{theorem} \label{thm-projection-coextensive-implies-strict-refinement}
If $X$ is a projection-coextensive object in a category $\C$ which admits (finite) products, then $X$ has the (finite) strict refinement property.
\end{theorem}
\begin{proof}
Suppose that  $(X \xrightarrow{f_i} A_i)_{i \in I}$ and $(X \xrightarrow{g_j} B_j)_{j \in J}$ are any two product diagrams for $X$. Let $\overline{A_n}$ be the product of the $A_i$'s where $n \neq i$ and let $\overline{f_n}:X \rightarrow \overline{A_n}$ be the induced morphism $(f_i)_{i\neq n}$, and similarly let $\overline{B_m}$ be the product of the $B_j$'s where $j \neq m$. Since $X$ is projection-coextensive, for each $n \in I$ and $m \in J$ there is a diagram
\[
\xymatrix{
	A_n \ar[d]_-{\alpha_{n,m}} &  X\ar[r]^{\overline{f_n}} \ar[d]^{g_m} \ar[l]_{f_n} & \overline{A_n}  \ar[d]^-{\overline{\alpha_{n,m}}}\\
	C_{n,m}  & \ar[r]_-{\beta'_{n,m}} B_m \ar[l]^-{\beta_{n,m}} & \overline{C_{n,m}}
}
\]
where each square is a pushout, and the bottom row is a product diagram. We show that the diagrams $(A_n \xrightarrow{\alpha_{n,j}} C_{n, j})_{j \in J}$ and $(B_m \xrightarrow{\beta_{i,m}} C_{i, m})_{i \in I}$ are product diagrams for any $n \in I$ and $m \in J$.  In the diagram 
\[
\xymatrix{
	A_n \ar[d]_-{(\alpha_{n,j})_{j \in J}} & &  X\ar[rr]^{\overline{f_n}} \ar[d]^{(g_j)_{j \in J}} \ar[ll]_{f_n} & & \overline{A_n}  \ar[d]^-{(\overline{\alpha_{n,j}})_{j \in J}} \\
	\prod \limits_{j \in J} C_{n,j}  & & \ar[rr]_-{\prod \limits_{j \in J}\beta'_{n,j}} \prod \limits_{j \in J}B_j \ar[ll]^-{\prod \limits_{j \in J}\beta_{n,j}} & & \prod\limits_{j\in J} \overline{C_{n,j}}
}
\]
the bottom row is a product diagram. Moreover, the outer vertical morphisms are product projections by Proposition~\ref{prop-composite-proposition}, and therefore by projection-coextensivity of $X$ the two squares above are pushouts. Since the central vertical morphism in the diagram is an isomorphism, it follows that the morphism $(\alpha_{n,j})_{j\in J}$ is an isomorphism, and we can similarly obtain $(\beta_{i, m})_{i \in I}$ as an isomorphism.
\end{proof}
\begin{remark}
An object $X$ is called \emph{directly-irreducible} if it is not a terminal object, and for any product diagram of $X$, one of the product projections in it is an isomorphism. The uniqueness of coproduct decompositions in an extensive category is well known (see Corollary 5.4 in \cite{Cockett1993}). Theorem~\ref{thm-projection-coextensive-implies-strict-refinement} together with Proposition~\ref{prop-projection-coextensive-closed-under-product-projections} allows us to deduce the same fact for product decompositions of projection-coextensive objects in a category with (finite) products: if $X$ is any projection-coextensive object in a category $\C$, and we are given two product diagrams $(a_i:X \rightarrow A_i)_{i\in I}$ and $(b_j:X \rightarrow B_j)_{j \in J}$ for $X$, where each $A_i$ and $B_j$ are directly-irreducible, then there exists a bijection $\sigma:I \rightarrow J$ and isomorphisms $f_i:A_i \rightarrow B_{\sigma(i)}$ such that $f_i a_i = b_{\sigma(i)}$.
\end{remark}
\begin{proposition} \label{prop-projection-coextensive-closed-under-product-projections}
	If $X$ is an $\M$-coextensive object in $\C$ and $p: X \rightarrow A$ any product projection of $X$, then $A$ is $\M$-coextensive. 
\end{proposition}
\begin{proof}
	Suppose that $A' \xleftarrow{p'} X \xrightarrow{p} A$ and $A_1 \xleftarrow{\pi_1} A \xrightarrow{\pi_2} A_2$ are both product diagrams, where $X$ is an $\M$-coextensive object. Consider the following diagram 
	\[
	\xymatrix{
		A_1 \times A' \ar[d]_{p_1} \ar@{}[dr] | {S_1} & X \ar[r]^-{\pi_2 p} \ar[l]_-{( \pi_1 p, p')} \ar[d]^p &  A_2 \ar[d]^{1_{A_{2}}} \ar[d] \ar@{}[dl] | {S_2} \\
		A_1 \ar[d]_{f_1} \ar@{}[dr] | {S_3} & A \ar[d]^f \ar[r]|{\pi_2} \ar[l]|{\pi_1} & A_2 \ar[d]^{f_2} \ar@{}[dl] | {S_4}\\
		B_1 & B \ar[r]|{\pi_2'} \ar[l]|{\pi_1'} & B_2 
	}
	\]
	where $(\pi_1 p, p')$ is the morphism induced into the product $A_1 \times A'$. Note that both $(\pi_1 p, p')$ and $\pi_1 p$ are product projections of $X$, and hence both are members of $\M$ by Remark~\ref{rem-prod-projection-in-M}. By Proposition~\ref{prop-composite-proposition}, $p_1$ is a member of $\M$ since $p_1(\pi_1 p, p') = \pi_1 p \in \M$. Since the top and middle rows are product diagrams and $p_1,p, 1_A$ are morphisms in $\M$, both the squares $S_1$ and $S_2$ are $\M$-pushout diagrams, since $X$ is $\M$-coextensive. Now, suppose that $f_1, f_2, f$ are morphisms from $\M$, then we will show that $S_3$ and $S_4$ are $\M$-pushouts if and only if the bottom row is a product diagram. If the bottom row is a product diagram, then both $S_1 + S_3$ and $S_2 + S_4$ are pushout diagrams by $\M$-coextensivity. It then follows from a general fact that since $S_1 + S_3$ and $S_1$ are pushouts, that $S_3$ is a pushout. Similarly, $S_4$ is a pushout. Conversely, if $S_3$ and $S_4$ are pushouts, then both $S_1 + S_3$ and $S_2 + S_4$ are pushouts, which implies that the bottom row is a product diagram by $\M$-coextensivity of $X$. Finally, the pushout of a morphism in $\M$ with domain $A$ along a product projection of $A$ always exists because the right-hand vertical morphism in $S_2$ is the identity and $f$ and $\pi_2$ were arbitrary. 
\end{proof}
\noindent
\section{A characterization of projection-coextensivity}
In what follows, we are aiming at showing that, under suitable conditions, $X$ is projection-coextensive if and only if $\Proj{X}$ is a Boolean lattice, which captures in a categorical way the classical characterization of the strict refinement property in \cite{CJT64}. The ``only if'' part of this statement is a direct analogue of one of the main results of Section 5 in \cite{Manes1992}, which shows that summands in a Boolean category (see Definition~\ref{def:Boolean-category} below for the dual) form a Boolean lattice. Proposition~\ref{prop-complements-unique} and Proposition~\ref{prop-order-reversing} are also analogues of Lemma 5.9 and the beginning part of the proof of Theorem~5.11, respectively,  in \cite{Manes1992}. We will conclude this section by showing that if $\C$ is a co-Boolean category, then every object is projection-coextensive, and that there are categories with finite products in which every object is projection-coextensive, and are not co-Boolean. However, if $\C$ has pushout stable product projections, then $\C$ is co-Boolean if and only if every object is projection-coextensive. 
\begin{definition}
Let $X$ be an object in a category $\C$, then by $\mathrm{Proj}_{\C}(X)$ we mean the full subcategory of $X\downarrow \C$ consisting of product projections of $X$. If $X$ has epimorphic product projections, then $\mathrm{Proj}_{\C}(X)$ is a preorder, so that in particular, if $X$ is projection-coextensive then $\mathrm{Proj}_{\C}(X)$ is a preorder (see Proposition~\ref{prop-product-projections-epimorphisms}). In what follows we shall write $\Proj{X}$ for the posetal-reflection of $\mathrm{Proj}_{\C}(X)$ when $X$ has epimorphic product projections. 
\end{definition}
\begin{remark} \label{rem-joins-of-projections}
Let $X$ be a projection-coextensive object in a category $\C$. If $X \xrightarrow{p} A$ is a product projection, then we shall write $[p]$ for the element of $\Proj{X}$ that $p$ represents. For product projections $p,q$ of $X$, note that $[p] \leqslant [q]$ if and only if $q$ factors through $p$. Note that the join $[p] \vee [q]$ exists, and is represented by the diagonal morphism in any pushout of $p$ along $q$. Then $\Proj{X}$ is bounded with top element $[X \rightarrow 1]$ and bottom element $[1_X]$.
\end{remark}
 
\begin{proposition} \label{prop-complements-unique}
Let $X$ be a projection-coextensive object in a category $\C$, and let $[p] \in \Proj{X}$ be any element. Then for any two complements $q,q'$ of $p$, we have $[q] = [q']$ in $\Proj{X}$. 
\end{proposition}
\begin{proof}
Suppose that $X \xrightarrow{p} A$ is any product projection with complements $X \xrightarrow{q} B$ and $X \xrightarrow{q'} B'$. By Proposition~\ref{prop-product-projections-epimorphisms} we may form the following diagram where every square is a pushout, and every edge a product diagram.
\[
\xymatrix{
\bullet & A \ar[r] \ar[l] & 1\\
A \ar[u] \ar[d]  & \ar[l]_{p} X \ar[d]^{q'} \ar[u]_{p} \ar[r]^{q} & B\ar[u] \ar[d]^{\beta} \\
1 & B' \ar[l] \ar[r]_{\alpha} & \bullet
}
\]
It then follows that $\alpha$ and $\beta$ are isomorphisms so that $[q] = [q']$. 
\end{proof}
\begin{corollary} \label{cor-terminal}
Given any projection-coextensive object $X$ in a category $\C$, and any product diagram $A \xleftarrow{a} X \xrightarrow{b} B$ in $\C$. If $a$ is an isomorphism then $B$ is a terminal object. 
\end{corollary}
\noindent
Let $X$ be a projection-coextensive object in a category $\C$, then for any $[p] \in \Proj{X}$  there exists unique \emph{complement} $[p]^{\bot} \in \Proj{X}$ with members of $[p]^{\bot}$ complements of members of $[p]$ (Proposition~\ref{prop-complements-unique}). We always denote the complement of $[p]$ by $[p]^{\bot}$ in what follows.
\begin{proposition} \label{prop-order-reversing}
For a projection-coextensive object $X$ in a category $\C$, the map $[p]\longmapsto [p]^{\bot}$ is order-reversing. 
\end{proposition}
\begin{proof}
The proof amounts to showing that in the diagram  
\[
\xymatrix{
	C_1 & B \ar[r] \ar[l]_{\gamma}  & C_2 \\
	A \ar@{..>}[ur]\ar[u] \ar[d]  & \ar[l]_{\pi_1} X \ar[d]^{\pi'_2} \ar[u]_{\pi'_1} \ar[r]^{\pi_2} & A'\ar[u] \ar[d]^{\beta} \\
	C_3 & B' \ar[l] \ar[r]_{\alpha} & C_4
}
\]
where the central column and central row are product diagrams and each square is a pushout, if the dotted arrow exists making the upper right-hand triangle commute, then $\pi_2$ factors through $\pi'_2$. Suppose that the dotted arrow exists, so that $\pi_1'$ factors through $\pi_1$. Then since $\pi_1$ is an epimorphism (Proposition~\ref{prop-product-projections-epimorphisms}), the upper left-hand triangle commutes, and therefore $\gamma$ is an isomorphism. The object $B$ is projection-coextensive by Proposition~\ref{prop-projection-coextensive-closed-under-product-projections}, so that by Corollary~\ref{cor-terminal} it follows that $C_2$ is terminal. Finally, this implies that $\beta$ is an isomorphism, since the right-hand edge is a product diagram, and hence $\pi_2$ factors through $\pi'_2$ via $\beta^{-1}$.
\end{proof}
\begin{corollary} \label{cor-projection-coextensive-implies-boolean-algebra}
Let $\C$ be a category and $X$ a projection-coextensive object in $\C$, then $\Proj{X}$ is a Boolean lattice.
\end{corollary}
\begin{proof}
By Proposition~\ref{prop-order-reversing} and Proposition~\ref{prop-complements-unique}, the map $[p]\longmapsto [p]^{\bot}$ is an order-reversing involution. This turns $\Proj{X}$ into a lattice, where meets are given by:
\[
[p] \wedge [q] = ([p]^{\bot} \vee [q]^{\bot})^{\bot}.
\]
 It is straight-forward to check that $\Proj{X}$ is uniquely complemented, i.e., for every $[p]$ if $[q]$ is such that $[p] \vee [q] = [X \rightarrow 1]$ and $[p]\wedge [q] = [1_X]$ then $[q] = [p]^{\bot}$, so that by Theorem~6.5 in \cite{Blyth2005}, it follows that $\Proj{X}$ is distributive, and hence a Boolean lattice.
\end{proof}

\begin{proposition} \label{prop-boolean-algebra-implies-projection-coextensive}
Suppose that $\C$ is a category, and let $X$ be an object with co-disjoint products, where the pushout of two product projections of $X$ are again product projections. If $\Proj{X}$ is a Boolean lattice, then $X$ is projection-coextensive. 
\end{proposition}
\begin{proof}
Note that the join of $[p]$ and  $[q]$ in $\Proj{X}$ is represented by the diagonal in any pushout of $p$ along $q$. Consider the diagram:
\[
\xymatrix{
	A_1 \ar[d]_{p_1} &  X \ar[dr]|{c_2} \ar[dl]|{c_1}\ar[r]^{a_2} \ar[d]^{p} \ar[l]_{a_1} & A_2  \ar[d]^{p_2} \\
	B_1  & \ar[r]_{b_2} B  \ar[l]^{b_1} & B_2
}
\]
where $a_1,a_2$ are complementary product projections, and $p_1,p,p_2$ are product projections. Note that $[a_1] \wedge [a_2] = [1_X]$ and $[a_1] \vee [a_2] = [X \rightarrow 1]$, since $X$ has co-disjoint products. Suppose that the two squares above are pushouts. Since $A_1 \xrightarrow{p_1} B_1$ and $A_2 \xrightarrow{p_2} B_2$ are product projections, the morphism $(c_1,c_2):X \rightarrow B_1 \times B_2$ is a product projection. By the universal property of product, it follows that $[p] \leqslant [(c_1,c_2)]$ in $\Proj{X}$. Also, we have $[(c_1,c_2)] \leqslant [c_1]$ and $[(c_1,c_2)] \leqslant [c_2]$, but since $[c_1] = [a_1] \vee [p]$ and $[c_2] = [a_2] \vee [p]$, it follows that
\[
[(c_1,c_2)] \leqslant ([a_1] \vee [p]) \wedge ([a_2] \vee [p]) = [p],
\]
and thus $[(c_1, c_2)] = [p]$ so that the bottom row is a product diagram. Now suppose that the bottom row is a product diagram, then we must have that $[c_1] \wedge [c_2] = [p]$. If $[c_1] = [X\rightarrow 1]$ then $b_2$ is an isomorphism, and since $a_2$ is an epimorphism, the right-hand square is a pushout. Moreover, $[a_2] \leqslant [p]$ so that $[a_1] \vee [p] = [X \rightarrow 1]$, so that the left-hand square is a pushout. Thus, in the case that $[c_1] = [X \rightarrow 1]$ both of the squares above are pushouts, and similarly if $[c_2] = [X \rightarrow 1]$. Now suppose that $[c_1], [c_2] < [X \rightarrow 1]$. Note that since $[a_1] \leqslant [c_1]$ and $[a_2] \leqslant [c_2]$ we have that $[c_1] \vee [c_2] = [X \rightarrow 1]$. In particular, this shows that $[c_1]$ is not comparable to $[c_2]$ (since if it were $[c_1] \vee [c_2]$ it would be either $[c_1]$ or $[c_2]$). 

Consider the elements $[a_1] \vee [p]$ and $[a_2] \vee [p]$. By the universal property of pushout, we have that $[a_1] \vee [p] \leqslant [c_1]$ and $[a_2] \vee [p] \leqslant [c_2]$. If $[a_1] \vee [p] \leqslant [c_2]$, then we would have that $[c_2] \geqslant ([a_1] \vee [p]) \vee ([a_2] \vee [p]) = [X \rightarrow 1]$, so that $[c_2] = [X \rightarrow 1]$, which contradicts the assumption, so that $[a_1] \vee [p]$ is not comparable to $c_2$ and similarly $[a_2] \vee [p]$ is not comparable to $[c_1]$. Since $[c_1] \wedge [c_2] = [p]$ and $([a_1] \vee [p]) \vee ([a_2] \vee [p]) = [X \rightarrow 1]$, the sublattice of $\Proj{X}$ generated by the elements $[a_1] \vee [p], [a_2] \vee [p], [p],[c_1],[c_2]$ is given by the Hasse diagram:
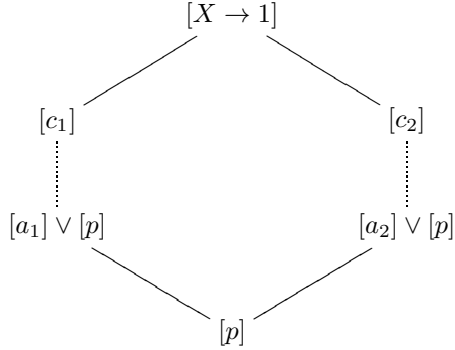
\[
\xymatrix{
& [X \rightarrow 1] & \\
[c_1] \ar@{-}[ur] & & [c_2] \ar@{-}[ul] \\
[a_1] \vee [p] \ar@{..}[u] & & [a_2] \vee [p] \ar@{..}[u] \\
& [p] \ar@{-}[ul] \ar@{-}[ur]& 
}
\]
where the solid lines indicate strict inequality. By Birkhoff's characterization of distributive lattices, $\Proj{X}$ cannot contain sublattices isomorphic to the pentagon. This forces $[c_1] = [a_1] \vee [p]$ and $[c_2] = [a_2] \vee [p]$. 
\end{proof}

\noindent
As a consequence of Corollary~\ref{cor-projection-coextensive-implies-boolean-algebra} and Proposition~\ref{prop-boolean-algebra-implies-projection-coextensive} we have the following characterization of projection-coextensivity. 
\begin{theorem} \label{thm-projection-coextensive-characterization}
Suppose that $\C$ is a category with finite products and let $X$ be any object with co-disjoint products, then the following are equivalent: 
\begin{enumerate}[(i)]
\item $X$ is projection-coextensive. 
\item The pushout of any two product projections of $X$ are product projections, and $\Proj{X}$ is a Boolean lattice.
\end{enumerate}
\end{theorem}
\noindent
The following definition is the dual of Definition~5.1 in \cite{Manes2006} (see also \cite{Manes1992} for the original definition of a Boolean category).
\begin{definition} \label{def:Boolean-category}
	A category $\C$ with finite products is \emph{co-Boolean} if it satisfies the following: 
	\begin{enumerate}
		\item For every product projection $p:X \rightarrow A$ and any morphism $f:X \rightarrow Y$ there exists a pushout square 
		\[
		\xymatrix{
	X \ar[r]^p \ar[d]_f  & A \ar[d] \\
	Y \ar[r]_q & P 	
	}
		\]
		with $q$ a product projection. 
		\item Product projections pushout product diagrams to product diagrams.
		\item If $X \xleftarrow{1_X} X \xrightarrow{1_X} X$ is a product diagram, then $X$ is a terminal object.
	\end{enumerate}
\end{definition}

\begin{corollary} \label{cor:Boolean-category}
	Every object in a co-Boolean category is projection-coextensive. 
\end{corollary}

\begin{proof}
Let $\C$ be a co-Boolean category, then by Lemma~5.2 in \cite{Manes1992} it follows that every object in $\C$ has co-disjoint products. Moreover, every $X$ in $\C$ has $\Proj{X}$ a Boolean-lattice by Theorem~5.11 in \cite{Manes1992}, so that by Theorem~\ref{thm-projection-coextensive-characterization}, every $X$ is projection-coextensive. 
\end{proof}

\begin{remark}
If $\C$ is a category with finite products where product projections are pushout stable, then the converse holds: $\C$ is co-Boolean if and only if every object is projection-coextensive. The necessity of the condition that product projections be pushout stable is illustrated by the full subcategory $\Grp_*$ of $\Grp$ consisting of \emph{centerless} groups. Recall that a group $G$ is centerless if its center $Z(G) = \{x\in G \mid \forall y \in G(xy = yx)\}$ is the trivial subgroup. A product of groups is centerless if and only if each factor is centerless. Every centerless group is projection-coextensive in $\Grp$ as well as in $\Grp_*$ (see Proposition~\ref{prop-centerless-implies-PC}). However, the category $\Grp_*$ is not a co-Boolean category. To see this, note that for any group $G$ the \emph{abelianization} of $G$ may be obtained from the pushout:
\[
\xymatrix{
G \ar[r]^{\Delta_G}  \ar[d] & G \times G \ar[d] \\
0 \ar[r] & G/[G,G]	
}
\]
If $G/[G,G]$ is non-trivial, then since $G/[G,G]$ is abelian, the morphism $G\times G \rightarrow G/[G,G]$ can not be a product projection, if $G$ is centerless. For example, if $G = S_3$ is the symmetric group on 3 elements, then its abelianization is isomorphic to $\mathbb{Z}/2\mathbb{Z}$, and also $S_3$ is centerless. This shows that centerless groups do not form a co-Boolean category, however they do form a category with finite products where every object is projection-coextensive. 
\end{remark}

\section{Projection-coextensivity in regular categories} \label{sec-projection-coextensivity-in-regular-categories}
A category $\C$ is \emph{regular} \cite{BGO71} if it has finite limits, coequalizers of kernel-pairs, and the pullback of a regular epimorphism along any morphism is again a regular epimorphism. Listed below are some elementary facts about morphisms in a regular category $\C$.

\begin{itemize}
\item Every morphism in $\C$ factors into a regular epimorphism followed by a monomorphism. 
\item If $f,g$ are regular epimorphisms in $\C$, then their product $f \times g$ is a regular epimorphism.
\item For any two morphisms $f:X \rightarrow Y$ and $g:Y \rightarrow Z$ in $\C$, if $g   f$ is a regular epimorphism, then $g$ is a regular epimorphism. Also, if $f$ and $g$ are regular epimorphisms, then $g  f$ is a regular epimorphism. 
\end{itemize}

\begin{definition}
	An object $X$ in a regular category $\C$ is said to have \emph{global support} if any terminal morphism $X \rightarrow 1$ is a regular epimorphism. 
\end{definition}

\begin{remark} \label{remark-product-projections-regular}
If $X$ is an object with global support in a regular category $\C$, then any product projection of $X$ is a regular epimorphism. This is because any factor of an object with global support itself has global support, and any product projection of $X$ can be obtained as a pullback of a terminal morphism of one of its factors. 
\end{remark}

\subsection{Subobjects and relations in regular categories} \label{sec-subobjects-relations-regular-categories}

Given any object $X$ in a regular category $\C$, consider the preorder of all monomorphisms with codomain $X$. The posetal-reflection of this preorder is $\Sub(X)$ -- the poset of \emph{subobjects} of $X$. For any morphism $f:X \rightarrow Y$ in $\C$ there is an induced Galois connection
\[
\xymatrix{
\Sub(X) \ar@/^1pc/[r]^{f_*}  &  \Sub(Y) \ar@/^1pc/[l]^{f^*}
}
\]
which is defined as follows. Given a subobject $A$ of $X$ represented by a monomorphism $A_0 \xrightarrow{a} X$, $f_*$ is defined to be the subobject represented by the monomorphism part of a (regular epi, mono)-factorization of $f   a$. Given a subobject $B$ of $Y$ represented by a monomorphism $B_0 \xrightarrow{b} Y$, we define $f^*(B)$ to be the subobject of $X$ represented by the monomorphism obtained from pulling back $b$ along $f$. A \emph{relation} from $X$ to $Y$ is a subobject of $X \times Y$, i.e., an isomorphism class of monomorphisms with codomain $X \times Y$. In regular categories we can \emph{compose} relations as follows: given a relation $R$ from $X$ to $Y$ and a relation $S$ from $Y$ to $Z$, and two representatives  $(r_1, r_2):R_0 \rightarrow X \times Y$ and $(s_1,s_2):S_0 \rightarrow Y \times Z$ of $R$ and $S$ respectively, form the pullback of $s_1$ along $r_2$:
\[
\xymatrix{
P \ar[r]^{p_2} \ar[d]_{p_1} & S_0 \ar[d]^{s_1} \\
R_0 \ar[r]_{r_2} & Y
}
\]
Then $R \circ S$ is defined to be the relation represented by the monomorphism part of any regular-image factorization of $(r_1 p_1, s_2p_2):P \rightarrow X \times Z$.
\begin{definition}
In what follows we will write $\Eq{X}$ for the poset of all equivalence relations on an object $X$. Recall that an equivalence relation $E$ is said to be \emph{effective} if it represented by a kernel pair of some morphism. We will write $\Ef{X}$ for the poset of all effective equivalence relations on an object $X$. 
\end{definition}
\begin{definition} \label{def-factor-relation}
	A \emph{factor relation} on an object $X$ is an equivalence relation $F$ represented by the kernel equivalence relation of a product projection $X \xrightarrow{p} A$ of $X$. A factor relation represented by the kernel equivalence relation of a complement of $p$ is correspondingly called a  \emph{complement} of $F$, and will often be denoted by $F'$. The poset $F(X)$ of factor relations on $X$ is bounded with top element $\nabla_X$ and bottom element $\Delta_X$.
\end{definition}
\noindent
For any effective equivalence relations $F,F'$ on an object $X$ in a regular category $\C$, we have $F \circ F' = \nabla_X$ if and only if the canonical morphism
\[
X \xrightarrow{(q_F,q_{F'})} X/F \times X/F'
\]
is a regular epimorphism (Proposition~\ref{prop-factor-relation-characterization}). The kernel equivalence relation of $(q_F,q_{F'})$ is given by $F \cap F'$, so that $(q_F, q_{F'})$ is mono if and only if $F \cap F' = \Delta_X$. These remarks are summarized in the following proposition, whose proof is omitted.
\begin{proposition}[Proposition 1.44. in \cite{HoePhd}] \label{prop-factor-relation-characterization}
A pair of effective equivalence relations $F,F'$ on an object $X$ in a regular category $\C$ are complementary factor relations if and only if $F \cap F' = \Delta_X$ and $F \circ F' = \nabla_X$. 
\end{proposition}
\begin{remark} \label{rem-F(X)-Boolean}
Suppose that $X$ is an object with global support in a regular category $\C$. Every element of $F(X)$ maps to an element of $\Proj{X}$ (by taking coequalizers), and likewise each element of $\Proj{X}$ maps to an element of $F(X)$ (by taking kernel pairs). These maps are inverse poset isomorphisms.
\end{remark}
\begin{proposition} \label{prop-half-projection-coextensive}
Let $\C$ be a regular category, and suppose that in the diagram
\[
\xymatrix{
	A_1 \ar[d]_{p_1} &  X\ar[r]^{\pi_2} \ar[d]|p  \ar[l]_{\pi_1} & A_2  \ar[d]^{p_2} \\
	B_1  & \ar[r]_{\pi'_2} Y  \ar[l]^{\pi'_1}  & B_2
}
\]
 the vertical morphisms are regular epimorphisms, the top and bottom rows are product diagrams, and that $X$ has global support, then the squares are pushouts.
\end{proposition}
\begin{proof} 
Consider the diagram
\[
\xymatrix{
	K_1 \ar@<-.5ex>[d]\ar@<.5ex>[d] & K \ar[r]^v\ar[l]_u \ar@<-.5ex>[d] \ar@<.5ex>[d] & K_2 \ar@<-.5ex>[d] \ar@<.5ex>[d]\\
	A_1 \ar[d]_{p_1} &  X\ar[r]^{\pi_2} \ar[d]|p  \ar[l]_{\pi_1} & A_2  \ar[d]^{p_2} \\
	B_1  & \ar[r]_{\pi'_2} Y  \ar[l]^{\pi'_1}  & B_2
}
\]
where $K_1, K$ and $K_2$ are the kernel-pairs of $p_1, p$ and $p_2$ respectively, and $u,v$ are the induced morphisms. The bottom row being a product diagram implies that the top row is a product diagram, and $X$ having global support implies that $K$ has global support, so that by Remark~\ref{remark-product-projections-regular} both $u$ and $v$ are (regular) epimorphisms. Then $u$ and $v$ being epimorphisms implies that the bottom squares are pushouts.
\end{proof}
\begin{theorem} \label{thm-projection-coextensive-equivalent-to-strict-refinement}
Suppose that $X$ is an object with global support in a regular category $\C$, and that $X$ admits pushouts of any two of its product projections. Then $X$ has the finite strict refinement property if and only if it is projection-coextensive. 
\end{theorem}
\begin{proof}
If $X$ is projection-coextensive, then by Theorem~\ref{thm-projection-coextensive-implies-strict-refinement} it follows that $X$ has the strict refinement property. Suppose that $A_1 \xleftarrow{\pi_1} X \xrightarrow{\pi_2} A_2$ is a product diagram, and that $X \xrightarrow{p} B$ is any product projection. If $X$ has the finite strict refinement property, then there exists product projections $p_1, p_2,b_1,b_2$ making the diagram
\[
\xymatrix{
	A_1 \ar[d]_{p_1} &  X\ar[r]^{\pi_2} \ar[d]|p \ar[l]_{\pi_1} & A_2  \ar[d]^{p_2} \\
	B_1  & \ar[r]_{b_2} Y  \ar[l]^{b_1} & B_2
}
\]
commute, and the bottom row a product diagram. Note that since $X$ has global support and $\C$ is regular, every morphism in the diagram above is a regular epimorphism. By Proposition~\ref{prop-half-projection-coextensive} the two squares are pushouts. Therefore, if we are given the diagram above, where the top row is a product diagram and the vertical morphisms are product projections, if the squares are pushouts then the bottom row is a product diagram. On the other hand, if the bottom row is a product diagram,  then by Proposition~\ref{prop-half-projection-coextensive} it follows that the two squares are pushouts. 
\end{proof}
\begin{corollary} \label{cor-weakest-coextensivity}
Given an object $X$ with global support in a regular category $\C$, then $X$ is projection-coextensive if and only if the pushout of a product diagram along a product projection exists, and is a product diagram.
\end{corollary}
\begin{corollary} \label{{cor-finite-SRP-implies-infinte-SRP}}
Let $\C$ be a complete regular category and let $X$ be an object with global support in $\C$. If $X$ has the finite strict refinement property, then $\C$ has the strict refinement property.
\end{corollary}
\begin{proof}
This follows from the fact that if $X$ has the finite strict refinement property, then it is projection-coextensive. Then $X$ being projection-coextensive, it has the strict refinement property by Theorem~\ref{thm-projection-coextensive-implies-strict-refinement}. 
\end{proof}
\subsection{Pre-exact categories and strict refinement}
Given a category $\C$ with kernel pairs and coequalizers of equivalence relations, any equivalence relation $E$ on any object $A$ in $\C$ admits an \emph{effective closure} $\overline{E}$, namely, the kernel equivalence relation of any coequalizer of $E$. 
\begin{definition} \label{def-pre-exact}
	A regular category $\C$ is said to be \emph{pre-exact} if it admits coequalizers of equivalence relations, and for any two equivalence relations $E_1$ and $E_2$, we have $\overline{E_1 \cap E_2} = \overline{E_1} \cap \overline{E_2}$.
\end{definition}
\begin{remark}
Any exact category $\C$ in the sense of \cite{BGO71} is pre-exact, since every equivalence relation $E$ is effective, and hence $E = \overline{E}$, so that  $\C$ satisfies the conditions of Definition~\ref{def-pre-exact}. 
\end{remark}
\begin{proposition} \label{prop-pre-exact}
	Let $\C$ and $\D$ be regular categories which have coequalizers of equivalence relations, and let $F: \C \rightarrow \D$ be a functor which preserves finite limits, coequalizers of equivalence relations, and reflects epimorphisms. If $\D$ is pre-exact, then so is $\C$. 
\end{proposition}
The proof of the above proposition is a standard preservation and reflection argument, the details of which can be found in the proof of Proposition 4.1 in \cite{HoePhd}. We include a sketch of that argument here:
\begin{proof}[Proof Sketch]
The functor $F$ preserves finite limits and effective closures of equivalence relations, and reflects equality of effective equivalence relation. The latter is due to the fact that any morphism in $\C$ between kernel pairs, which is an epimorphism, is necessarily and isomorphism. 
\end{proof}
\begin{example} \label{example-pre-exact-categories}
	The dual categories  $\Top^{\op}, \Ord^{\op}, \Grph^{\op}, \Rel^{\op}$ of topological spaces, ordered sets, graphs and binary relations, all admit forgetful functors to $\Set^{\op}$ these forgetful functors satisfy the conditions of Proposition~\ref{prop-pre-exact}, and since $\Set^{\op}$ is pre-exact (since it is exact), it follows that each of the categories above are pre-exact (since they are all regular categories which have coequalizers of equivalence relations). The category of topological groups $\Grp(\Top)$ is regular, but not exact. It has all small limits and colimits, and the forgetful functor $\Grp(\Top) \rightarrow \Grp$ has both a left and a right adjoint, and therefore it preserves all limits and colimits which exist in $\Grp(\Top)$. This functor reflects epimorphisms, and therefore since $\Grp$ is exact, $\Grp(\Top)$ is pre-exact.
\end{example}
\noindent
The next result is an analogue of Theorem~4.5 in \cite{CJT64}, for regular categories (see also Theorem~5.17 in \cite{MMT87}). Note that, given two factor relations $F$ and $G$ on an object $X$ in a pre-exact category $\C$ such that $F \circ G = G \circ F$, the composite $F \circ G$ is an equivalence relation which is the join of $F$ and $G$ in $\Eq{X}$. Moreover, the join of $F$ and $G$ in the lattice $\Ef{X}$ of effective equivalence relations exists, and is given by the effective closure $\overline{F\circ G}$ of $F \circ G$. In what follows we will write $F \overline{\circ} G$ for $\overline{F\circ G}$. 
\begin{proposition} \label{prop-characterization-strict-refinement}
	The following are equivalent for an object $X$ with global support in a pre-exact category $\C$. 
\begin{enumerate}[(i)]
	\item $X$ is projection-coextensive.
	\item $F(X)$ is a sublattice of $\Ef{X}$ which is Boolean. 
	\item $F(X)$ forms a Boolean lattice under the operations $\overline{\circ}$ and $\cap$.
\end{enumerate}
\end{proposition}
\begin{proof}
 $(i) \implies (ii)$:  By Corollary~\ref{cor-projection-coextensive-implies-boolean-algebra}, if $X$ is projection-coextensive, then $\Proj{X}$ is a Boolean lattice where joins are given by pushout. Note that $F(X)$ is isomorphic to $\Proj{X}$ (see Remark~\ref{rem-F(X)-Boolean}), and the join of two factor relations $F$ and $G$ is given by their join in $\Ef{X}$. We next show that the meet of $F$ and $G$ in $F(X)$ is given by their meet $F \cap G$ in the lattice $\Ef{X}$ of effective equivalence relations on $X$. Suppose that $F', G'$ are the complementary factor relations of $F$ and $G$ respectively. Since $X$ is projection-coextensive, every edge in the outer square:
\[
\xymatrix{
	\frac{X}{F \vee G} &  \frac{X}{F} \ar[r] \ar[l] &  \frac{X}{F \vee G'} \\
	\frac{X}{G } \ar[u] \ar[d] &  X \ar[u] \ar[d] \ar[l] \ar[r]  &  \frac{X}{G'} \ar[u] \ar[d] \\
	\frac{X}{F' \vee G} &  \frac{X}{F'} \ar[r]\ar[l] &  \frac{X}{F' \vee G'} \\
}
\] 
is a product diagram. This implies that
\[
(F \vee G) \cap (F' \vee G) = G \quad \text{and} \quad (F \vee G) \cap (F \vee G') = F.
\]  
Since the canonical morphism 
\[
X \longrightarrow \frac{X}{F \vee G} \times  \frac{X}{F' \vee G} \times \frac{X}{F \vee G'}
\]
is a complementary product projection of $X \rightarrow \frac{X}{F' \vee G'}$, and by Proposition~\ref{prop-complements-unique} complements are unique, it follows that the meet $F \wedge G$ of two factor relations in $F(X)$ is given by:
\[
F \wedge G = (F' \vee G')' = (F \vee G) \cap (F' \vee G) \cap (F \vee G') = F \cap G.
\]
Therefore $F(X)$ is a sublattice of $\Ef{X}$ which is Boolean. $(ii) \implies (iii):$ Suppose that $F,G \in F(X)$, and that $F'$ is the complement of $F$ and $G'$ the complement of $G$. Then
\[
((F \circ G \circ F) \cap F ') \subseteq (F \vee G) \cap F' \subseteq G,
\]
which implies that 
\[
F \circ G \circ F \subseteq  (F \circ G \circ F) \cap (F ' \circ F) = ((F \circ G \circ F) \cap F ') \circ F \subseteq G \circ F.
\]
This implies that  $F \circ G$ is an equivalence relation, so that $\overline{F \circ G} = F \vee G$ in $F(X)$, and therefore $F(X)$ is a Boolean lattice under $\overline{\circ}$ and $\cap$. For $(iii) \implies (i)$, just note that $F(X)$ being Boolean under $\overline{\circ}$ and $\cap$ implies that $\Proj{X}$ is a Boolean lattice, and that the pushout of any two product projections of $X$ are product projections.  Therefore, by Theorem~\ref{thm-projection-coextensive-characterization} it follows that $X$ is projection-coextensive, since every object with global support in a regular category has co-disjoint-products.
\end{proof}
\begin{remark}
Note that the proof of $(ii) \implies (iii)$ is similar to the proof of Theorem 4.5 $(vi) \implies (v)$ in \cite{CJT64}. 
\end{remark}

In \cite{Bou05}, the author characterized the congruence distributivity property for regular Goursat categories in terms of preservation of binary meets of equivalence relations by regular-epimorphisms. One of the basic observations is that if $f:X \rightarrow Y$ is any regular epimorphism in a regular category $\C$ and $E$ any equivalence relation on $X$, then in the notation of Section~\ref{sec-subobjects-relations-regular-categories} we have
\[
(f \times f)^* (f \times f)_* (E) = K \circ E \circ K,
\]
where $K$ is the kernel equivalence relation of $f$. In what follows we will denote $f^{-1}(E) = (f \times f)^*(E)$ and $f(E) = (f\times f)_*(E)$, so that the above equation reduces to:
\[
f^{-1}(f(E)) = K \circ E \circ K.
\]
\begin{proposition} \label{prop-characterization-two-strict-refinement}
	The following are equivalent for an object $X$ with global support in a pre-exact category $\C$. 
	\begin{enumerate}[(i)]
		\item $X$ is projection-coextensive.
		\item For any $F,G \in F(X)$, $F \circ G = G \circ F$ and we have 
		\[
		q(G) \cap q(G') = \Delta_{X/F},
		\]
		where $q:X \rightarrow X/F$ is a canonical quotient.
	\end{enumerate}
\end{proposition}
\begin{proof}
Suppose that $X$ is projection-coextensive, so that by Proposition~\ref{prop-characterization-strict-refinement}, $F(X)$ is a Boolean lattice under $\overline{\circ}$ and $\cap$. Let $q:X \rightarrow X/F$ be a canonical quotient, then:
\begin{align*}
q(G) \cap q(G') = q(q^{-1}(q(G) \cap q(G'))) &= q(q^{-1}(q(G)) \cap q^{-1}(q(G'))) \\
&= q((F \circ G) \cap (F\circ G')) \\
&\subseteq q((F \overline{\circ} G) \cap (F \overline{\circ} G'))\\
& = q(F) = \Delta_{X/F}.
\end{align*}
For the converse, suppose $F,G \in F(X)$ and that $q:X \rightarrow X/F$ is a canonical quotient, then since $q(G) \cap q(G') = \Delta_{X/F}$ it follows that
\[
F  = q^{-1}(q(G) \cap q(G')) = q^{-1}q(G) \cap q^{-1}q(G')  \implies F = (F \circ G) \cap (F \circ G').
\]
Since 
\[
\nabla_X = F \circ  G \circ F' \circ  G \subseteq \overline{F \circ  G} \circ \overline{F' \circ  G},
\] 
the canonical morphism 
\[
X \rightarrow (X/\overline{F \circ  G}) \times (X/\overline{F' \circ  G}),
\] 
is a regular epimorphism. Its kernel equivalence relation is given by
\[
\overline{F' \circ  G} \cap \overline{F \circ  G} = \overline{(F' \circ  G) \cap (F \circ  G)} = G,
\]
since $\C$ is pre-exact. This implies that pushing out the product diagram $X/F \leftarrow X \rightarrow X/F'$ along $X \rightarrow X/G$ is a product diagram, and thus by Corollary~\ref{cor-weakest-coextensivity} it follows that $X$ is projection-coextensive. 
\end{proof}
\noindent
As a summary of  Theorem~\ref{thm-projection-coextensive-implies-strict-refinement}, Theorem~\ref{thm-projection-coextensive-equivalent-to-strict-refinement}, Propositions~\ref{prop-characterization-strict-refinement} and \ref{prop-characterization-two-strict-refinement} we have the following theorem.
\begin{theorem} \label{thm-characterization-two-strict-refinement}
	The following are equivalent for an object $X$ with global support in a small complete pre-exact category $\C$. 
	\begin{enumerate}[(i)]
		\item $X$ has the strict refinement property.
		\item $X$ has the finite strict refinement property.
		\item $X$ is projection-coextensive.
		\item $F(X)$ is a sublattice of $\Ef{X}$ which is Boolean. 
		\item $F(X)$ is a Boolean lattice under the operations $\overline{\circ}$ and $\cap$.
		\item  For any $F,G \in F(X)$, $F \circ G = G \circ F$ and we have 
		\[
		q(G) \cap q(G') = \Delta_{X/F}
		\]
		where $q:X \rightarrow X/F$ is a canonical quotient, and $G'$ is the complement of $G$.
	\end{enumerate}
\end{theorem}
\section{Majority categories have strict refinements}
The notion of a \emph{majority category} was first introduced and studied in \cite{Hoe18a} and \cite{Hoe18b}. The theorem below slightly updates Theorem~3.1 in \cite{Hoe18b}, where the equivalence of $(i)$, $(ii)$, $(iii)$ and $(v)$ of the theorem below was shown.

\begin{theorem} \label{thm-characterization-majority-categories}
	For a regular category $\C$, the following are equivalent.
	\begin{enumerate}[(i)]
\item For any three reflexive relations $R,S,T$ on any object $X$ in $\C$ we have $R \cap (S \circ T) \leqslant (R \cap S) \circ (R \cap T)$.
\item For any three reflexive relations $R,S,T$ on any object $X$ in $\C$ we have $R \circ(S \cap T) \geqslant (R \circ S) \cap (R \circ T)$.
\item For any three equivalence relations $A,B,C$ on any object $X$ in $\C$ we have $A \cap (B \circ C) = (A \cap B) \circ (A \cap C)$.
\item For any three equivalence relations $A,B,C$ on any object $X$ in $\C$ we have $A \circ (B \cap C) = (A \circ B) \cap (A \circ C)$.
\item $\C$ is a majority category.
	\end{enumerate}
\end{theorem}
\begin{proof}
It is easily seen that $(ii)$ implies $(iv)$, thus it suffices show that $(iv)$ implies $(iii)$. We note that using a similar argument as in the proof of Theorem~3.1 in \cite{Hoe18b}, it can be shown that if $(iv)$ holds, then the formula $(B \cap C) \circ A = (B \circ A) \cap (C \circ A)$ also holds for any equivalence relations $A,B,C$ on any object $X$ in $\C$. Suppose that $(iv)$ holds, and that $A,B,C$ are any equivalence relations on the same object in $\C$, then:
\begin{align*}
(A \cap B) \circ (A \cap C) &= ((A \cap B) \circ A) \cap ((A \cap B) \circ C) \\
 &= A \cap (A \circ C) \cap (B \circ C) \\
  &= A \cap (B \circ C).
\end{align*}
This shows that $(iii)$ holds. 
\end{proof}
In \cite{Gra04} the notion of a \emph{factor permutable} category was introduced, and as we will see, every regular majority category is factor permutable. 
\begin{definition}
	A regular category $\C$ is said to be \emph{factor permutable} if for any equivalence relation $E$ on any object $X$ in $\C$ we have $F \circ E = E \circ F$ for any factor relation $F \in F(X)$. 
\end{definition}
\begin{proposition} \label{prop-majority-factor-permuatable}
	Every regular majority category $\C$ is factor-permutable.
\end{proposition}
\begin{proof}
	Let $E$ be any equivalence relation on an object $X$ in $\C$, and let $F$ be a factor relation on $X$ with complement $F'$. Then 
	\[
	E \circ F = (E \circ F) \cap (F \circ F') \leqslant ((E \circ F) \cap F) \circ ((E \circ F) \cap F') \leqslant F \circ (E \cap F') \circ (F \cap F') \leqslant F \circ E
	\]
	by Theorem~\ref{thm-characterization-majority-categories}.
\end{proof}

\begin{proposition}
Every object $X$ with global support in a pre-exact majority category is projection-coextensive.
\end{proposition}
\begin{proof}
 Given any two factor relations $F,G$ with complements $F',G'$ respectively, the two equivalence relations $F$ and $G$ permute by Proposition~\ref{prop-majority-factor-permuatable}, so that the composite $F \circ G$ is an equivalence relation. Moreover we have 
\[
\overline{F \circ G} \cap (F'\cap G') = \overline{(F \circ G )\cap (F'\cap G')} = \Delta_X,
\]
but also, we have
\[
\nabla_X = (F \circ G) \circ (F' \cap G') \subseteq \overline{F \circ G} \circ (F'\cap G').
\]
Therefore, $\overline{F \circ G}$ and $F' \cap G'$ are complementary factor relations, so that $F(X)$ is closed under the operations of $\overline{\circ}$ and $\cap$. It then follows by Theorem~\ref{thm-characterization-majority-categories} that $F(X)$ is distributive, so that by Theorem~\ref{thm-characterization-two-strict-refinement}, $X$ is projection-coextensive.
\end{proof}
\section{Centerless objects in a Mal'tsev category have strict refinements}
\noindent
Recall the notion of a \emph{Mal'tsev category}: 
\begin{definition}[\cite{CLP91, CPP91}]
A finitely complete category $\C$ is Mal'tsev if any reflexive relation is an equivalence relation.
\end{definition}
In \cite{BG02}, the authors develop a categorical approach to centrality of equivalence relations. The central notion, that of a \emph{connector} of equivalence relations, is defined in any finitely-complete category. When the base category is Mal'tsev, it reduces to the following:
\begin{definition}[\cite{BG02}]
Let $\C$ be a Mal'tsev category, and let $(r_1,r_2):R_0 \rightarrow X\times X$ and $(s_1,s_2):S_0 \rightarrow X \times X$ represent two equivalence relations $R$ and $S$ on an object $X$, respectively. Consider the pullback below:
\[
\xymatrix{
R_0 \times_X S_0 \ar[r]^-{p_1} \ar[d]_-{p_2} &  S_0 \ar[d]^{s_1} \\
R_0 \ar[r]_{r_2} & X
}
\]
A \emph{connector} between $R_0$ and $S_0$ is a morphism $p:R_0 \times_X S_0 \rightarrow X$ such that $x R p(x,y,z) S z$ and 
\[
p(x,y,y) = x = p(y,y,x).
\]
If there exists a representative of $R$ which admits a connector with a representative of $S$, then we say that $R$ and $S$ are \emph{connected}, or that the pair $(R,S)$ is commuting. 
\end{definition}
\noindent
Let $\C$ be a regular Mal'tsev category, then for any equivalence relations $E,F,G,H$ on any object $X$ in $\C$, the following properties hold:
\begin{itemize}
\item  If $E \leqslant F$ and $(F,G)$ is commuting, then $(E,G)$ is commuting (Proposition~3.10 in \cite{BG02}). 
\item  If $(E,F)$ is commuting and $f:X \rightarrow Y$ is any regular epimorphism, then $(f(E), f(F))$ is commuting (Proposition~4.2 in \cite{BG02}). 
\item  If $(E,F)$ is commuting, and $(E,G)$ is commuting, then $(E,F\vee G)$ is commuting (Proposition~4.3 in \cite{BG02}). 
\item If $(E,G)$ is commuting, and $(F,H)$ is commuting, then $(E \times F, G \times H)$ is commuting (Proposition 3.12 in \cite{BG02}).
\item If $E \cap F = \Delta_X$, then $(E,G)$ is commuting (Lemma~3.9 in \cite{BG02}).
\end{itemize}
\begin{definition}
An object $X$ in a regular Mal'tsev category $\C$ is said to be \emph{centerless} if for any equivalence relation $E$ on $X$, we have that if $(E, \nabla_X)$ is commuting then $E = \Delta_X$. 
\end{definition}
\begin{proposition} \label{prop-product-preserves-centerless}
Let $X$ and $Y$ be objects in a regular Mal'tsev category $\C$, then $X \times Y$ is centerless if and only if both $X$ and $Y$ are centerless. 
\end{proposition}
\begin{proof}
Suppose that $X \times Y$ is centerless, and let $E$ be any equivalence relation on $X$ such that $(E,\nabla_X)$ is commuting. Trivially, we have that $(\Delta_Y, \nabla_Y)$ is commuting, so that $(E \times \Delta_Y, \Delta_X \times \nabla_Y)$ is commuting. Therefore $E \times \Delta_Y = \Delta_{X\times Y}$, which implies $E = \Delta_X$. Conversely, if $X$ and $Y$ are centerless then given any equivalence relation $E$ on $X\times Y$ such that $(E,\nabla_{X \times Y})$ is commuting, then $(\pi_1(E), \nabla_X)$ and $(\pi_2(E), \nabla_Y)$ is commuting, so that $\pi_1(E) = \Delta_X$ and $\pi_2(E) = \Delta_Y$. This implies that $E = \Delta_{X \times Y}$. 
\end{proof}
\begin{proposition} \label{prop-centerless-implies-PC}
If $X$ is a centerless object with global support in a pre-exact Mal'tsev category, then $X$ is projection-coextensive.
\end{proposition}
\begin{proof}
We will show that $X$ satisfies the conditions of $(vi)$ in Theorem~\ref{thm-characterization-two-strict-refinement}. Suppose that $F,G \in F(X)$ are any factor relations, and let $G'$ be the corresponding complement of $G$. Then we have that $F \circ G = G \circ F$, since regular Mal'tsev categories are congruence permutable \cite{CPP91}. Let $q:X \rightarrow X/F$ be a canonical quotient morphism, then by Proposition~\ref{prop-product-preserves-centerless}, it follows that $X/F$ is centerless. Since $G \cap G' = \Delta_X$ it follows that the pair $(G,G')$ is commuting, and therefore $(q(G), q(G'))$ is commuting, so that both $(q(G)\cap q(G'), q(G))$ and $(q(G)\cap q(G'), q(G'))$ is commuting, and consequently $(q(G)\cap q(G'), q(G') \vee q(G))$ is commuting. Since $q(G') \vee q(G) = \nabla_{X/F}$, we have that $(q(G)\cap q(G'), \nabla_{X/F})$ is commuting. This implies $q(G)\cap q(G') = \Delta_{X/F}$ since $X/F$ is centerless. 
\end{proof}

\section{Regular-coextensivity and factorable congruences}
Suppose that $X$ is any object in a category $\C$ with products. Given any product diagram $A \xleftarrow{\pi_1} X \xrightarrow{\pi_2} B$, and any effective equivalence relations $E$ and $F$ represented by $e:E_0 \rightarrow A^2$ and $f:F_0 \rightarrow B^2$ then the composite monomorphism
\[
E_0 \times F_0 \xrightarrow{e \times f} A^2 \times B^2 \xrightarrow{\tau} X^2
\]
represents an effective equivalence relation on $X$ which will always be denoted by $E \times F$ in what follows.
\begin{definition}\label{def-factorable-congruences}
	An object $X$ in a category with finite products is said to have \emph{factorable congruences} if for any effective equivalence relation $E$ on $X$, and any product diagram $A \xleftarrow{\pi_1} X \xrightarrow{\pi_2} B$, there exist effective equivalence relations $E_1$ and $E_2$ on $A$ and $B$ respectively, such that $E_1 \times E_2 = E$.
\end{definition}
\begin{definition}
An object $X$ in a category $\C$ is said to be \emph{regularly-coextensive} if it is $\M$-coextensive with $\M$ the class of all regular epimorphisms in $\C$.
\end{definition}
\begin{proposition} \label{prop-factorable-congruences-reformulation}
	Let $\C$ be a regular category, then any object $X$ with global support is regularly-coextensive if and only if for any regular epimorphism $X \xrightarrow{q} Q$, and any binary product diagram $A \xleftarrow{\pi_1} X \xrightarrow{\pi_2} B$ there exist morphisms $q_1:A \rightarrow Q_1$ and $q_2:B \rightarrow Q_2$   and morphisms $p_1: Q \rightarrow Q_1$ and $p_2: Q \rightarrow Q_2$ such that in the commutative diagram
\[
\xymatrix{
	A \ar[d]_{q_1} &  X\ar[r]^{\pi_2} \ar[d]^q  \ar[l]_{\pi_1} & B \ar[d]^{q_2} \\
	Q_1  & \ar[r]_{p_2} Q  \ar[l]^{p_1} & Q_2
}
\]
the bottom row is a product diagram. 
\end{proposition}
\noindent
The proof below is similar to the proof of Theorem~\ref{thm-projection-coextensive-equivalent-to-strict-refinement}. 
\begin{proof}
The ``only if'' part is immediate. Consider the diagram in the statement of the proposition, then by Proposition~\ref{prop-half-projection-coextensive} it follows that the two squares are pushouts. This also implies that pushing out the top row along $q$ produces a product diagram. On the other hand, if the bottom row is a product diagram, then by Proposition~\ref{prop-half-projection-coextensive} both squares are pushouts.
\end{proof}
\begin{proposition} \label{prop-characteriziation-regular-coextensive}
Let $\C$ be a regular category, then any object $X$ with global support has factorable congruences if and only if it is regularly-coextensive.
\end{proposition}
\begin{proof}
Let $X$ have factorable congruences, and suppose that $E$ is any effective equivalence relation, and that $X \xrightarrow{f} Q$ is a quotient of $X$ by $E$. Suppose also that $A \xleftarrow{\pi_1} X \xrightarrow{\pi_2} B$ is any product diagram. By assumption, there exist effective equivalence relations $E_1$ and $E_2$ on $A$ and $B$ respectively with $E_1 \times E_2 = E$. Let $f_1:A \rightarrow Q_1$ and $f_2:B \rightarrow Q_2$ be the respective quotients of $E_1$ and $E_2$, then there exist $Q \xrightarrow{q_1} Q_1$ and $Q \xrightarrow{q_2} Q_2$ such that the bottom row in the diagram 
\[
\xymatrix{
	A \ar[d]_{f_1} &  X\ar[r]^{\pi_2} \ar[d]^f  \ar[l]_{\pi_1} & B \ar[d]^{f_2} \\
	Q_1  & \ar[r]_{q_2} Q  \ar[l]^{q_1} & Q_2
}
\]
is a product diagram. Thus by Proposition~\ref{prop-factorable-congruences-reformulation} it follows that $X$ is regularly-coextensive.
\end{proof}
\noindent 
Recall from Remark~\ref{remark-product-projections-regular} that in a regular category $\C$, every object with global support has regular-epimorphic product projections. This immediately gives the following corollary.  
\begin{corollary}\label{cor-reg-implies-proj}
If $X$ is an object with global support in a regular category $\C$, which is regularly-coextensive then $X$ is projection-coextensive.
\end{corollary}
\begin{remark}
In \cite{I2000} the author shows, amongst other things, that if a non-empty universal algebra has factorable congruences, then it has the strict refinement property. Corollary~\ref{cor-reg-implies-proj}, generalizes this result. 
\end{remark}
\begin{remark}
We note that a similar proof of Proposition~2.7 in \cite{Hoe19} (see also Remark~2.8 in \cite{Hoe19}) may be used to show that if $\C$ is a category in which every object is regularly coextensive, then binary products commute with arbitrary coequalizers in $\C$.  This shows in particular, that if $\C$ is a variety of universal algebra with constants, in which every algebra has factorable-congruences, then binary products commute with arbitrary coequalizers in $\C$. 
\end{remark}
\section{Concluding remarks} \label{sec-concluding-remarks}
If $\M$ is a class of morphisms closed under products in $\C$, then we could have defined $X$ to be $\M$-coextensive when for any product diagram $A \xleftarrow{\pi_1} X \xrightarrow{\pi_2} B$ the canonical functor 
\[
(\M \downarrow A) \times (\M \downarrow B) \rightarrow (\M \downarrow X)
\]
is an equivalence (where $\M\downarrow Y$ is the full subcategory of $\C\downarrow Y$ consisting of morphisms from $\M$). Under suitable conditions on $\M$ this notion could be equivalently reformulated in a similar way as Definition~\ref{def-M-coextensive}, where ``$\M$-pushout'' would have a different meaning. Still further variants of coextensivity are possible, and it is an interesting question of whether or not other refinement properties, such as the \emph{intermediate refinement property} \cite{CJT64}, are captured as some variant of coextensivity.

\end{document}